\documentclass[11pt]{amsart}

\usepackage{tikz}
\usetikzlibrary{matrix,arrows}

\usepackage{mathrsfs}
\usepackage{amsmath, amsthm, amsfonts}
\usepackage[percent]{overpic}
\usetikzlibrary{matrix,arrows}

\usepackage{xypic}
\usepackage{color}
\usepackage{MnSymbol}
\usepackage{fullpage}
\usepackage{bbm}

\usepackage{thmtools}
\usepackage{thm-restate}

\usepackage{hyperref}

\usepackage{cleveref}

\declaretheorem[name=Theorem,numberwithin=section]{thm}

\newtheorem*{thm*}{Theorem}

\newtheorem{prop}[thm]{Proposition}
\newtheorem{lem}[thm]{Lemma}

\theoremstyle{definition}
\newtheorem{defn}[thm]{Definition}

\newtheorem{ex}[thm]{Example}
\theoremstyle{remark}
\newtheorem{rem}[thm]{Remark}

\newcommand{\QQ}{\mathbb{Q}}

\newcommand{\HH}{\mathcal{H}}
\newcommand{\PPP}{\mathcal{P}}

\newcommand{\RR}{\mathbb{R}}
\newcommand{\CC}{\mathbb{C}}

\newcommand{\OO}{\mathcal{O}}
\newcommand{\ord}{\operatorname{ord}}

\newcommand{\PP}{\mathbb{P}}

\newcommand{\Pic}{\text{Pic}}
\newcommand{\Mg}{\mathcal{M}_g}
\newcommand{\Mgn}{\mathcal{M}_{g,n}}
\newcommand{\Mgb}{\overline{\mathcal{M}}_g}
\newcommand{\Mgnb}{\overline{\mathcal{M}}_{g,n}}

\newcommand{\Eff}{\overline{\text{Eff}}}

\newcommand{\rk}{\text{rk}}

\newcommand{\res}{\text{res}}

\newcommand{\SL}{\mbox{SL}}

\newcommand{\Div}{\operatorname{Div}}

\newtheorem{Thm*}{Theorem*}
\theoremstyle{definition}

\title{On the effective cone of $\overline{\mathcal{M}}_{g,n}$
}
\author{Scott Mullane\\
}
\date{\today}

\AtEndDocument{\bigskip{\footnotesize%
\textsc{Scott Mullane, Department of Mathematics, Harvard University, One Oxford St, Cambridge, MA 02138, USA} \par
\textit{E-mail address}: \texttt{smullane@math.harvard.edu} \par
}}

\begin{document}
\thispagestyle{empty}

\maketitle

\begin{abstract}
For every $g\geq 2$ and $n\geq g+1$ we exhibit infinitely many extremal effective divisors in $\overline{\mathcal{M}}_{g,n}$ coming from the strata of abelian differentials.
\end{abstract}

\setcounter{tocdepth}{1}

\tableofcontents

\section{Introduction}
The space of holomorphic differentials or one-forms on a curve provides the canonical map of the curve into projective space. Hence the differentials encode a wealth of information about a curve. Instead of considering all differentials on a fixed curve, we can extract the information held by differentials from a different perspective. By fixing the type of differential and also allowing the curve to vary, we obtain the moduli space of abelian differentials $\HH(\kappa)$ consisting of pairs $(C,\omega)$ where $\omega$ is a holomorphic or meromorphic differential on a smooth curve $C$ and the multiplicity of the zeros and poles of $\omega$ is fixed of type $\kappa$, an integer partition of $2g-2$.  Equivalently, a flat metric on $C$ with a finite number of singularities gives a translation atlas on the curve away from these singularities. Hence a pair $(C,\omega)$ is also known as a translation surface and  can be represented as polygons in the plane with parallel side identifications. Previous seminal work exposes the fundamental algebraic attributes of these spaces~\cite{McMullen}\cite{KontsevichZorich}\cite{EskinMirzakhani}\cite{EskinMirzakhaniMohammadi}\cite{Filip}.

Subvarieties of codimension one are the vanishing locus of a global section of a line bundle and are known as effective divisors. Geometrically defined effective divisors of Brill-Noether and Petri type were first used by~\cite{HarrisMumford}\cite{Harris}\cite{EisenbudHarrisKodaira} to show that $\Mgb$ is of general type for $g>23$. This was extended by Koszul divisors to $g\geq 22$ by~\cite{Farkas23}\cite{FarkasKoszul}. Similar results were obtained by~\cite{Logan} for $\Mgnb$ with $n>0$, which were again sharpened by~\cite{FarkasKoszul}.

For any fixed holomorphic signature $\kappa$, the locus of curves admitting a one-form of type $\kappa$ forms a subvariety in $\Mg$. In~\cite{Mullane} the author presents a closed formula for the divisor class in $\Mgb$ of the closure of all possible codimension one loci of this type. This generalised the two known cases of the loci of curves with an extremal Weierstrass point~\cite{Diaz} and the loci of curves with a subcanonical pencil~\cite{Teixidor}. In $\Mgnb$, effective divisors from the strata of abelian differentials have been previously used under different guises in a number of places, for example, divisors of Brill-Noether type~\cite{Logan}, the pullback of the theta divisor~\cite{Muller}\cite{GruZak}, the loci of points in the support of an odd theta characteristic~\cite{FarkasBN} and the anti-ramification locus~\cite{FarkasVerraTheta}. In~\cite{Mullane2} the author uses the degeneration of abelian differentials~\cite{BCGGM} and maps between moduli spaces to compute many new classes of divisors of this type as well as efficiently reproducing these previously known classes.

The stratum of canonical divisors with signature $\kappa=(k_1,...,k_m)$ is a subvariety of $\mathcal{M}_{g,m}$,
\begin{equation*}
\PPP(\kappa):=\{[C,p_1,...,p_m]\in {\mathcal{M}}_{g,m}   \hspace{0.15cm}| \hspace{0.15cm}k_1p_1+...+k_mp_m\sim K_C\}.
\end{equation*}
The moduli space of twisted canonical divisors of type $\kappa$, a closure of this variety in $\overline{\mathcal{M}}_{g,m}$, was provided by~\cite{FP} who showed that this space contained the main component coming from canonical divisors on smooth curves and extra components completely contained in the boundary of $\overline{\mathcal{M}}_{g,m}$. The global residue condition that distinguishes the main component was given by~\cite{BCGGM}, providing a full compactification of $\PPP(\kappa)$. 

Forgetting points to obtain a codimension one subvariety and taking the closure we obtain the divisor 
\begin{equation*}
D^n_\kappa=\overline{\{[C,p_1,...,p_n]\in\Mgn\hspace{0.15cm}|\hspace{0.15cm} [C,p_1,...,p_m]\in\mathcal{M}_{g,m} \text{ with  }  \sum k_ip_i\sim K_C   \}},
\end{equation*}
in $\Mgnb$ for $\kappa=(k_1,...,k_m)$ with $\sum k_i=2g-2$ where $m=n+g-2$ or $n+g-1$ for holomorphic and meromorphic signature $\kappa $ respectively. For $g\geq 2$ this divisor is irreducible unless all $k_i$ are even where there are two irreducible components based on spin structure that we denote by the indices \emph{odd} and \emph{even}. In \S\ref{divisorclass} we review the known classes $D^n_\kappa$ for meromorphic signature $\kappa$ due to~\cite{Muller}\cite{GruZak}\cite{Mullane2}.

Tensoring global sections of any two line bundles $L_1,L_2$ we obtain a global section of the tensor product $L_1\otimes L_2$. Hence the sum of two effective divisors on a projective variety $X$ is an effective divisor and there is a natural cone structure on effective divisors on $X$. Taking the closure of this cone we obtain $\Eff(X)$, the pseudo-effective cone of $X$. A fundamental problem in the birational geometry of moduli spaces is to understand the structure of this cone. The complete structure of the pseudo-effective cone for $\Mgnb$ remains an open problem.  ~\cite{CastravetTevelevHypertree} found finitely many extremal effective divisors in $\overline{\mathcal{M}}_{0,n}$ for every $n\geq 7$ indexed by irreducible hypertrees.  \cite{ChenCoskun} gave infinitely many extremal effective divisors in $\overline{\mathcal{M}}_{1,n}$ for every $n\geq3$, showing in these cases the pseudo-effective cone cannot be rational polyhedral. In general genus $g\geq 2$, ~\cite{FarkasVerraU}\cite{FarkasVerraTheta} gave a finite number of extremal divisors for any fixed $g$ and $n$ with $g-2\leq n\leq g$.

A moving curve $B$ in a projective variety $X$ is a curve class such that irreducible curves with numerical class equal to $B$ cover a Zariski dense subset of $X$. Hence $B\cdot D\geq 0$ for any pseudo-effective divisor $D$, further, if $B\cdot D=0$ then $D$ lies on the boundary of the pseudo-effective cone. A covering curve $B$ of an effective divisor $D$ is a curve class such that irreducible curves with numerical class equal to $B$ cover a Zariski dense subset of $D$. If $B\cdot D<0$ then $D$ is extremal in the pseudo-effective cone.

Taking a fibration of $\overline{\PPP}(\kappa)$ for a meromorphic signature $\kappa$ we obtain curves in $\Mgnb$,
\begin{equation*}
B^n_{\kappa}:=\overline{\bigl\{ [C,p_1,...,p_n]\in\mathcal{M}_{g,n} \hspace{0.15cm}\big| \hspace{0.15cm} \text{fixed general $[C,p_{g+2},...,p_m]\in\mathcal{M}_{g,m-g-1}$ and } \sum_{i=1}^m k_ip_i \sim K_C   \bigr\}   }.
\end{equation*}
For $m=|\kappa|\geq n+g$ these curves provide covering curves for $\Mgnb$. In \S\ref{covering} we show these curves yield the following information about the pseudo-effective cone.

\begin{restatable}{thm}{boundary}
\label{thm:boundary}
For $g\geq 2$, meromorphic signature $\kappa$ with $|\kappa|=n+g-1$ for $n\geq g+1$,
\begin{equation*}
B^n_{\kappa,1,-1}\cdot D^n_\kappa=0.
\end{equation*}
Hence the irreducible components of the divisors $D^n_\kappa$ lie on the boundary of the pseudo-effective cone. Further, if all $k_i$ are even then 
\begin{equation*}
aD^{n,\text{odd}}_{\kappa}+bD^{n,\text{even}}_{\kappa}
\end{equation*}
for any $a,b\geq 0$, lies on the boundary of the pseudo-effective cone.
\end{restatable}

For any divisor $D^n_\kappa$, curves with numerical class equal to $B^n_\kappa$ cover a Zariski dense subset of the divisor. For certain signatures we are able to show the curve $B^n_\kappa$ is a component of a specialisation of $B^n_{\kappa,1,-1}$ with the other component being completely contained in the boundary of $\Mgnb$. The positive intersection of these boundary curves with $D^n_\kappa$ gives 
\begin{equation*}
B^n_\kappa\cdot D^n_\kappa<0,
\end{equation*}
showing these divisors to be extremal and rigid.

\begin{restatable}{thm}{extremal}
\label{thm:extremal}
The divisors $D^{g+1}_{\underline{d},1^{g-1}}$ in $\overline{\mathcal{M}}_{g,g+1}$ for $g\geq 2$ are rigid and extremal for $\underline{d}=(d_1,d_2,d_3,1^{g-2})$ with $d_1+d_2+d_3=1$ and $\sum_{d_i<0}d_i\leq-2$. Hence $\Eff(\Mgnb)$ is not rational polyhedral for $g\geq 2$, $n\geq g+1$.
\end{restatable}


A $\QQ$-factorial projective variety $X$ is a Mori dream space if it is of the simplest type from the perspective of the minimal model program. Introduced by~\cite{HuKeel}, if the Cox ring of $X$ is finitely generated and the Neron-Severi group of $X$ is $\Pic(X)\otimes\QQ$ then the nef and semi-ample cones of divisors on $X$ correspond and the effective cone of divisors on $X$ is rational polyhedral and has a specific chamber decomposition. By providing a nef divisor that is not semi-ample, \cite{Keel} showed $\Mgnb$ is not a Mori dream space for $g\geq 3$ and $n\geq 1$.
\cite{ChenCoskun} exhibited infinitely many extremal divisors in the effective cone of $\overline{\mathcal{M}}_{1,n}$ for $n\geq3$ showing that these spaces are not Mori dream spaces.
$\overline{\mathcal{M}}_{0,n}$ was long thought to be a Mori dream space due to its similarity to a toric variety. This was known for $n\leq 6$, however,~\cite{CastravetTevelev} showed $\overline{\mathcal{M}}_{0,n}$ is not a Mori dream space for $n\geq 134$. This result was improved to $n\geq13$ by~\cite{GK} and then to $n\geq 10$ by~\cite{HKL} who further showed that the intermediate open cases would require a new method. Theorem~\ref{thm:extremal} sheds light on the open cases of genus $g=2$, providing the following result.

\begin{restatable}{cor}{moridreamspace}
\label{cor:moridreamspace}
$\overline{\mathcal{M}}_{2,n}$ is not a Mori dream space for $n\geq 3$.
\end{restatable}
$\text{}$
\\
\textbf{Acknowledgements.} I am grateful to Dawei Chen, Joe Harris and Anand Patel for many stimulating discussions on topics related to this paper.

\section{Preliminaries}
\subsection{Strata of abelian differentials}\label{strata:abelian}
The \emph{stratum of abelian differentials with signature $\kappa=(k_1,...,k_n)$,} a non-zero integer partition of $2g-2$ is defined as
\begin{equation*}
\HH(\kappa):=\{ (C,\omega) \hspace{0.15cm}|\hspace{0.15cm} g(C)=g,\hspace{0.05cm} (\omega)=k_1p_1+...+k_np_n, \text{ for $p_i$ distinct}\}
\end{equation*}
where $\omega$ is a meromorphic differential on $C$. Hence $\HH(\kappa)$ is the space of abelian differentials with prescribed multiplicities of zeros and poles given by $\kappa$. By relative period coordinates $\HH(\kappa)$ has dimension $2g+n-1$ if $\kappa$ is holomorphic (all $k_i>0$) and $2g+n-2$ if $\kappa$ is meromorphic (some $k_i<0$). 

The \emph{stratum of canonical divisors with signature $\kappa$} is defined as
\begin{equation*}
\PPP(\kappa):=\{[C,p_1,...,p_n]\in \Mgn   \hspace{0.15cm}| \hspace{0.15cm}k_1p_1+...+k_np_n\sim K_C\}.
\end{equation*}
By forgetting the ordering of the zeros or poles of the same multiplicity we obtain the projectivisation of $\HH(\kappa)$ under the $\CC^*$ action that scales the differential $\omega$. Hence $\PPP(\kappa)$ has dimension $2g+n-2$ or $2g+n-3$ for holomorphic or meromorphic $\kappa$ respectively. 

A \emph{theta characteristic} $\eta$ is a line bundle on a smooth curve $C$ such that $\eta^{\otimes 2}\sim K_C$. The parity of $h^0(C,\eta)$ is known as the \emph{spin structure} of $\eta$ and is deformation invariant~\cite{Mumford}\cite{Atiyah}. If all $k_i$ are even then any $(C,\omega)$  with $(\omega)=\sum_{i=1}^nk_ip_i$ specifies a theta characteristic on the underlying curve
\begin{equation*}
\eta\sim \sum_{i=1}^n \frac{k_i}{2}p_i.
\end{equation*}
In these cases the loci $\HH(\kappa)$ and $\PPP(\kappa)$ are reducible with disjoint components with even and odd parity of $h^0(C,\eta)$ known as spin structure. 

For specific signatures an extra component appears as differentials on a hyperelliptic curves resulting from pulling back a degree $g-1$ rational function under the unique hyperelliptic cover of $\PP^1$.~\cite{KontsevichZorich} showed that this resulted in a distinct hyperelliptic component for signatures $\kappa=(2g-2)$ and $(g-1,g-1)$ for $g\geq 4$. Hence $\HH(\kappa)$ and hence $\PPP(\kappa)$ has at most $3$ connected components  for holomorphic $\kappa$.~\cite{Boissy} showed that for $g\geq 2$ in the meromorphic case, in addition to signatures with all $k_i$ even, signatures with all even positive entries and poles of the form $\{-1,-1\}$ also had two spin structure components based on the parity of the line bundle
\begin{equation*}
\eta\sim \sum_{k_i>0} \frac{k_i}{2}p_i.
\end{equation*}
Further, $\HH(\kappa)$ and hence $\PPP(\kappa)$ contain a hyperelliptic component distinct from any possible spin structure component if the zeros of $\kappa$ are of the form $\{2n\}$ or $\{n,n\}$ for any positive integer $n$ and the poles are of the form $\{-2p\}$ or $\{-p,-p\}$ for any positive integer $p\geq 2$.

\subsection{Degeneration of abelian differentials}\label{degen}
The moduli space of twisted canonical divisors of type $\kappa=(k_1,...,k_n)$ was defined by~\cite{FP}. Consider a stable pointed curve $[C,p_1,...,p_n]\in\Mgnb$. 
A twisted canonical divisor of type $\kappa$ is a collection of (possibly meromorphic) divisors $D_j$ on each irreducible component $C_j$ of $C$ such that
\begin{enumerate}
\item
The support of $D_j$ is contained in the set of marked points and the nodes lying in $C_j$, moreover if $p_i\in C_j$ then $\ord_{p_i}(D_j)=k_i$.
\item
If $q$ is a node of $C$ and $q\in C_i\cap C_j$ then $\ord_q(D_i)+\ord_q(D_j)=-2$.
\item
If $q$ is a node of $C$ and $q\in C_i\cap C_j$ such that $\ord_q(D_i)=\ord_q(D_j)=-1$ then for any $q'\in C_i\cap C_j$ we have $\ord_{q'}(D_i)=\ord_{q'}(D_j)=-1$. We write $C_i\sim C_j$.
\item
If $q$ is a node of $C$ and $q\in C_i\cap C_j$ such that $\ord_q(D_i)>\ord_q(D_j)$ then for any $q'\in C_i\cap C_j$ we have $\ord_{q'}(D_i)>\ord_{q'}(D_j)$. We write $C_i\succ C_j$.
\item
There does not exist a directed loop $C_1\succeq C_2\succeq...\succeq C_k\succeq C_1$ unless all $\succeq$ are $\sim$.
\end{enumerate}  
In addition to the main component containing canonical divisors of type $\kappa$ on smooth pointed curves, ~\cite{FP} showed this space included extra components completely contained in the boundary.
The global residue condition, recently provided by~\cite{BCGGM}, distinguishes the main component from the boundary components giving a full compactification for the strata of abelian differentials. Let $\Gamma$ be the dual graph of $C$. A twisted canonical divisor of type $\kappa$ is the limit of twisted canonical divisors on smooth curves if there exists a collection of meromorphic differentials $\omega_i$ on $C_i$ with $(\omega_i)=D_i$ that satisfy the following conditions
\begin{enumerate}
\item
If $q$ is a node of $C$ and $q\in C_i\cap C_j$ such that $\ord_q(D_i)=\ord_q(D_j)=-1$ then $\res_q(\omega_i)+\res_q(\omega_j)=0$.
\item
There exists a full order on the dual graph $\Gamma$, written as a level graph $\overline{\Gamma}$, agreeing with the order of $\sim$ and $\succ$, such that for any level $L$ and any connected component $Y$ of  $\overline{\Gamma}_{>L}$ that does not contain a prescribed pole we have
\begin{equation*}
\sum_{\begin{array}{cc}\text{level}(q)=L, \\q\in C_i\in Y\end{array}}\res_{q}(\omega_i)=0
\end{equation*} 
\end{enumerate}
Part (b) is known as the \emph{global residue condition}. We provide two explicit examples of the global residue condition that will be relevant to later sections.

\begin{ex}\label{Ex1}
Consider a genus $g=2$ curve $C$ with a $\PP^1$-tail $X$, attached at a node $x$, a general point in $C$. Let $j\geq 2$. Figure 1 depicts a twisted canonical divisor of type $\kappa=(j,-j,1,1)$ on this nodal curve with
\begin{eqnarray*}
D_C&=&x+p_4\sim K_C,\\
D_{X}&=&jp_1-jp_2+p_3-3x\sim K_{X}.
\end{eqnarray*}
\begin{figure}[htbp]
\begin{center}
\begin{overpic}[width=0.6\textwidth]{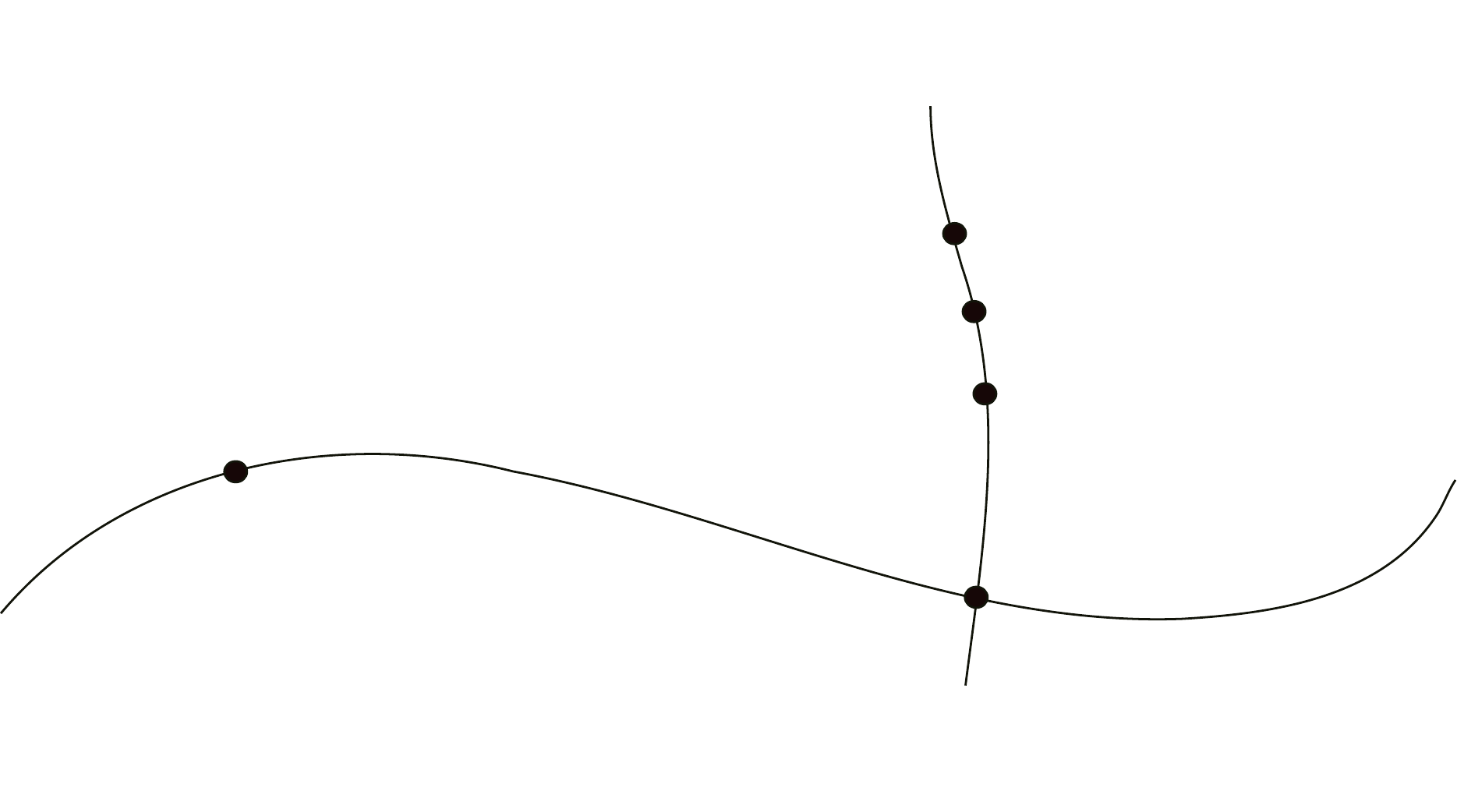}

\put (10,24){$$}
\put (36,24){$$}
\put (12,20){$p_4$}
\put (63,12){$x$}
\put (62,17){$$}
\put (61,28){$p_3$}
\put (60.5,33.5){$p_2$}
\put (60,39){$p_1$}

\put (-5,16){$C$}

\put (65,50){$X\cong \PP^1$}

\put (-5,-0){\footnotesize{Figure 1: A twisted canonical divisor of type $\kappa=(j,-j,1,1)$ on a nodal curve. }}
\end{overpic}
\end{center}
\end{figure}
To find the conditions on such a twisted canonical divisor being smoothable we consider the conditions placed on the residues by the possible level graphs. In this case there are only two components and there is just one possible level graph pictured in Figure 2. The global residue condition then gives that the condition on the twisted differential being the limit of differentials on smooth curves is that there exists a differential $\omega_{X}$ with $(\omega_{X})\sim D_{X}$ and $\res_x(\omega_X)=0$.
\begin{figure}[htbp]
\begin{center}
\begin{overpic}[width=0.25\textwidth]{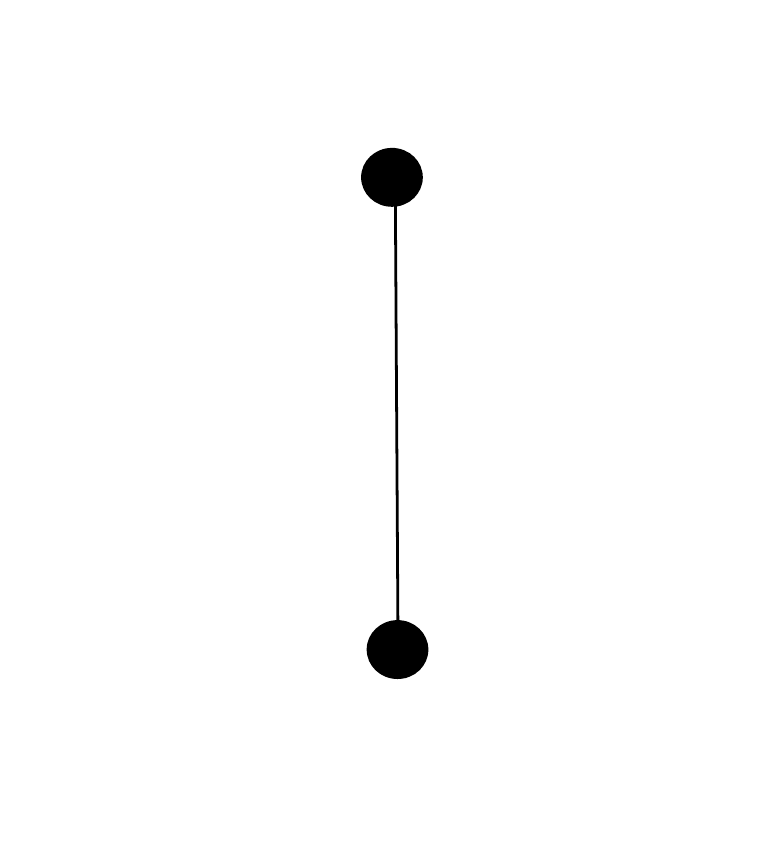}
\put (32, 23){$X$}
\put (32, 78){$C$}
\put (-50,-0){\footnotesize{Figure 2: The level graph giving the global residue condition.}}
\end{overpic}
\end{center}
\end{figure}

Consider the differential $\omega_{X}$. By the cross ratio we can set the poles to $0$ and $\infty$ and the zeros to $1$ and $a$. The resulting differential is given locally at $0$ by
\begin{equation*}
c\frac{(z-1)^{j}(z-a)}{z^{3}}dz
\end{equation*}  
for some constant $c\in \CC^*$. The residue at $0$ is
\begin{equation*}
c(-1)^{j-1}\left(j+\begin{pmatrix} j\\ 2 \end{pmatrix}a\right),
\end{equation*}
which is zero for the unique value $a=-2/(j-1)$. 
Hence though any value of $a$ gives a twisted canonical divisor, it is only for this unique value of $a$ that the twisted canonical divisor is the limit of canonical divisors of this type on smooth curves.
\begin{figure}[htbp]
\begin{center}
\begin{overpic}[width=0.5\textwidth]{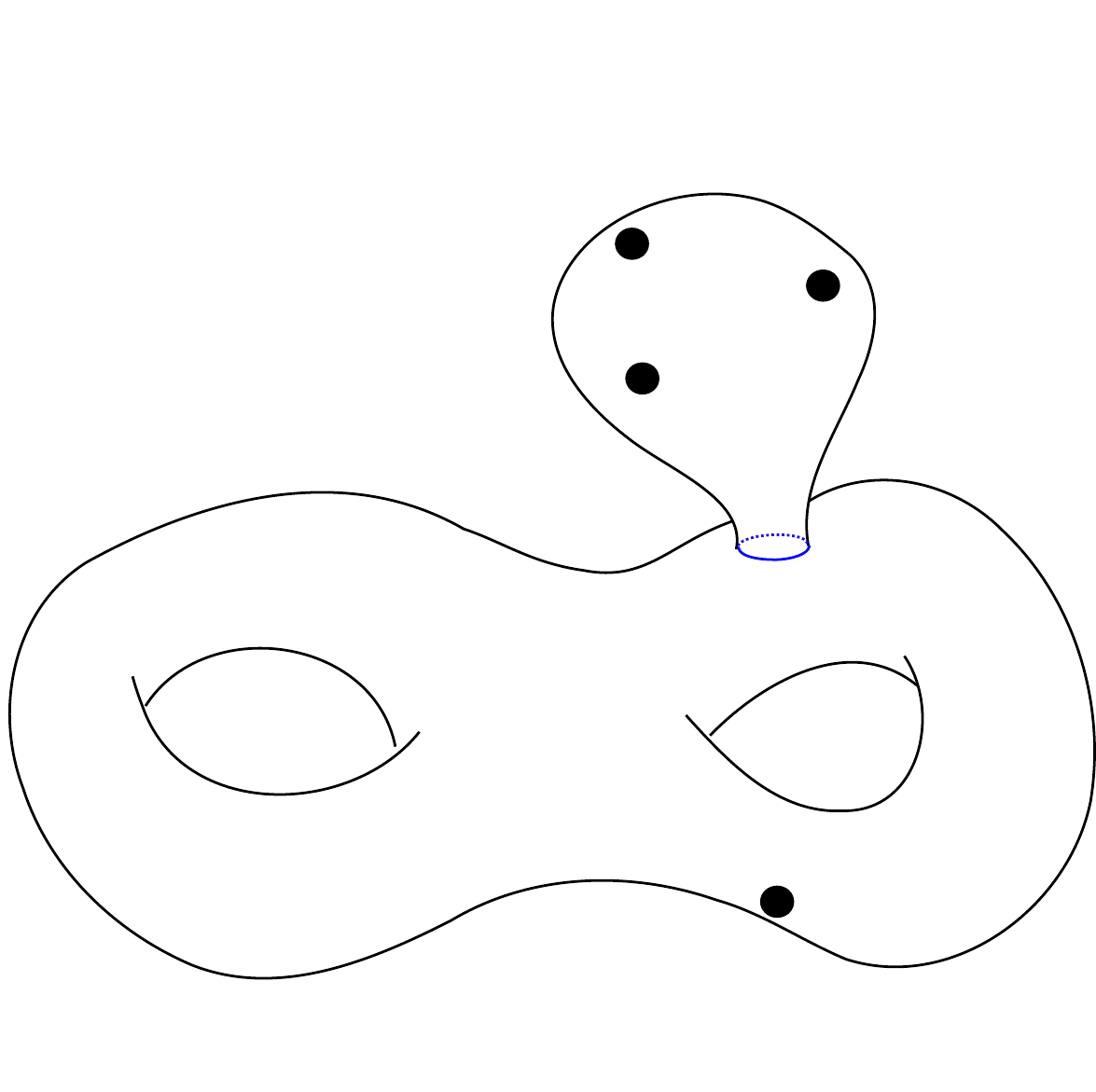}
\put (83, 73){$X_t$}
\put (0,16){$C_t$}
\put (70,45){$v$}
\put (61,76){$p_1$}
\put (62,64){$p_2$}
\put (68,73){$p_3$}
\put (74,16){$p_4$}

\put (-15,-0){\footnotesize{Figure 3: A twisted canonical divisor of type $\kappa=(j,-j,1,1)$ on a smooth curve. }}
\end{overpic}
\end{center}
\end{figure}

The necessity of the condition on the residue in this case can be observed by considering topologically a family of twisted canonical divisors on smooth curves degenerating to a nodal marked curve. Let $\chi$ be a family of meromorphic differentials $(Z_t,\omega_t)$ with $\omega_t$ of type $\kappa=(j,-j,1,1)$ on smooth curves $Z_t$ for $t\ne 0$, degenerating to the nodal curve $Z_0=C\cup_x X$ at $t=0$. Figure 3 shows $(Z_t,\omega_t)$ for $t\ne 0$. Let $v$ be the vanishing cycle on $Z_t$ that shrinks to the node $x$ at $t=0$. Observe $v$ separates the curve into components $C_t$ and $X_t$ with $C_t\to C$ and $X_t\to X$ as $t\to 0$. Now as $C_t$ contains no poles, by an application of Stokes formula to the cycle $v$ we have
\begin{equation*}
\int_v \omega_t=0.
\end{equation*}
The global residue condition is the limit of this condition as $t\to 0$ and hence we have shown that the condition is necessary. Complex-analytic plumbing techniques and flat geometry are used in~\cite{BCGGM} to show the global residue condition obtained in this way is sufficient in all cases.

\end{ex}

\begin{ex}\label{Ex2}
The following example provides further practice working with the global residue condition and will also be relevant to our later calculations. 
\begin{figure}[htbp]
\begin{center}
\begin{overpic}[width=0.5\textwidth]{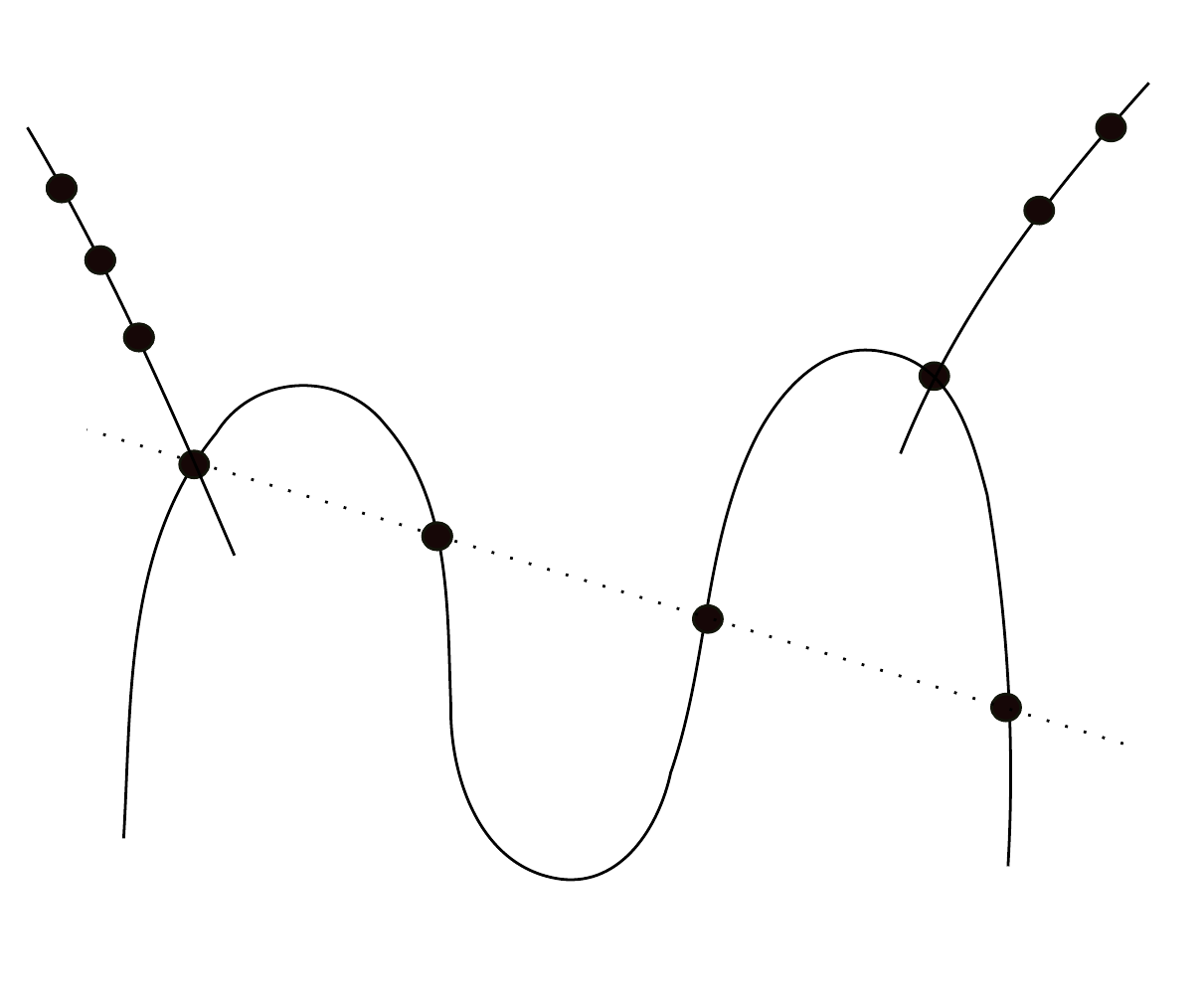}
\put (4,16){$C$}
\put (20,62){$Y\cong \PP^1$}
\put (8,68){$p_1$}
\put (11,62){$p_2$}
\put (14,55){$p_3$}
\put (10,43){$y$}
\put (31,36){$p_4$}
\put (53,28){$p_5$}
\put (78,22){$p_6$}
\put (85,56){$X\cong \PP^1$}
\put (80,66){$p_7$}
\put (86,73){$p_8$}
\put (76,55){$x$}
\put (26,0){$y+p_4+p_5+p_6\sim K_C$}

\put (-30,-10){\footnotesize{Figure 4: A twisted canonical divisor of type $\kappa=(d_1,d_2,d_3,1^4,-1)$ on a nodal genus $g=3$ curve.}}
\end{overpic}
\vspace{0.8cm}
\end{center}
\end{figure}
Consider a genus $g=3$ curve $C$ with $\PP^1$-tails we denote $X$ and $Y$ attached at a nodes $x$ and $y$ respectively. Figure 4 depicts a twisted canonical divisor of type $\kappa=(d_1,d_2,d_3,1^4,-1)$ with $\sum d_i=1$ for $d_i\ne 0$, $\{d_i\}\ne\{1,1,-1\}$  on this nodal curve with
\begin{eqnarray*}
D_C&=&y+p_4+p_5+p_6 \sim K_C\\
D_X&=&p_{7}-p_{8}-2x\sim K_X\\
D_Y&=&d_1p_1+d_2p_2+d_3p_3-3y\sim K_Y.
\end{eqnarray*}
To find the conditions on such a twisted canonical divisor being smoothable we consider the conditions placed on the residues by the possible level graphs.  
\begin{figure}[htbp]
\begin{center}
\vspace{1cm}
\begin{overpic}[width=0.7\textwidth]{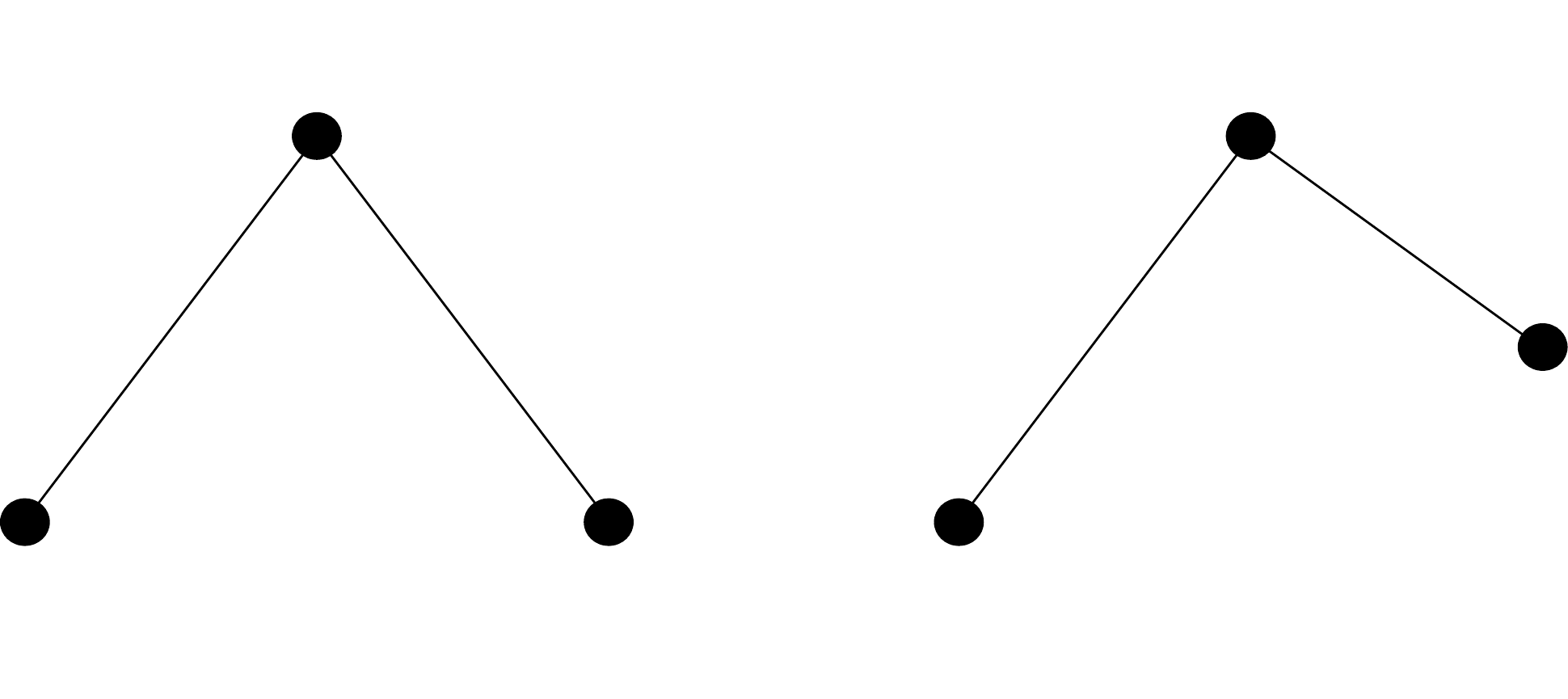}

\put (19,40){$C$}
\put (0,6){$X$}
\put (38,6){$Y$}
\put (13,47){Graph A}

\put (78.5,40){$C$}
\put (59,6){$X$}
\put (97,17){$Y$}
\put (72,47){Graph B}

\put (11,-0){\footnotesize{Figure 5: Level graphs giving the global residue condition.}}
\end{overpic}
\end{center}
\end{figure}
Observe that due to the simple pole at $p_{8}$, for any differential $\omega_X$ on $X$ with $(\omega_X)\sim p_{7}-p_{8}-2x$ we have $\res_x(\omega_X)=-\res_{p_{8}}(\omega_X)\ne 0$. 
\begin{figure}[htbp]
\begin{center}
\begin{overpic}[width=0.7\textwidth]{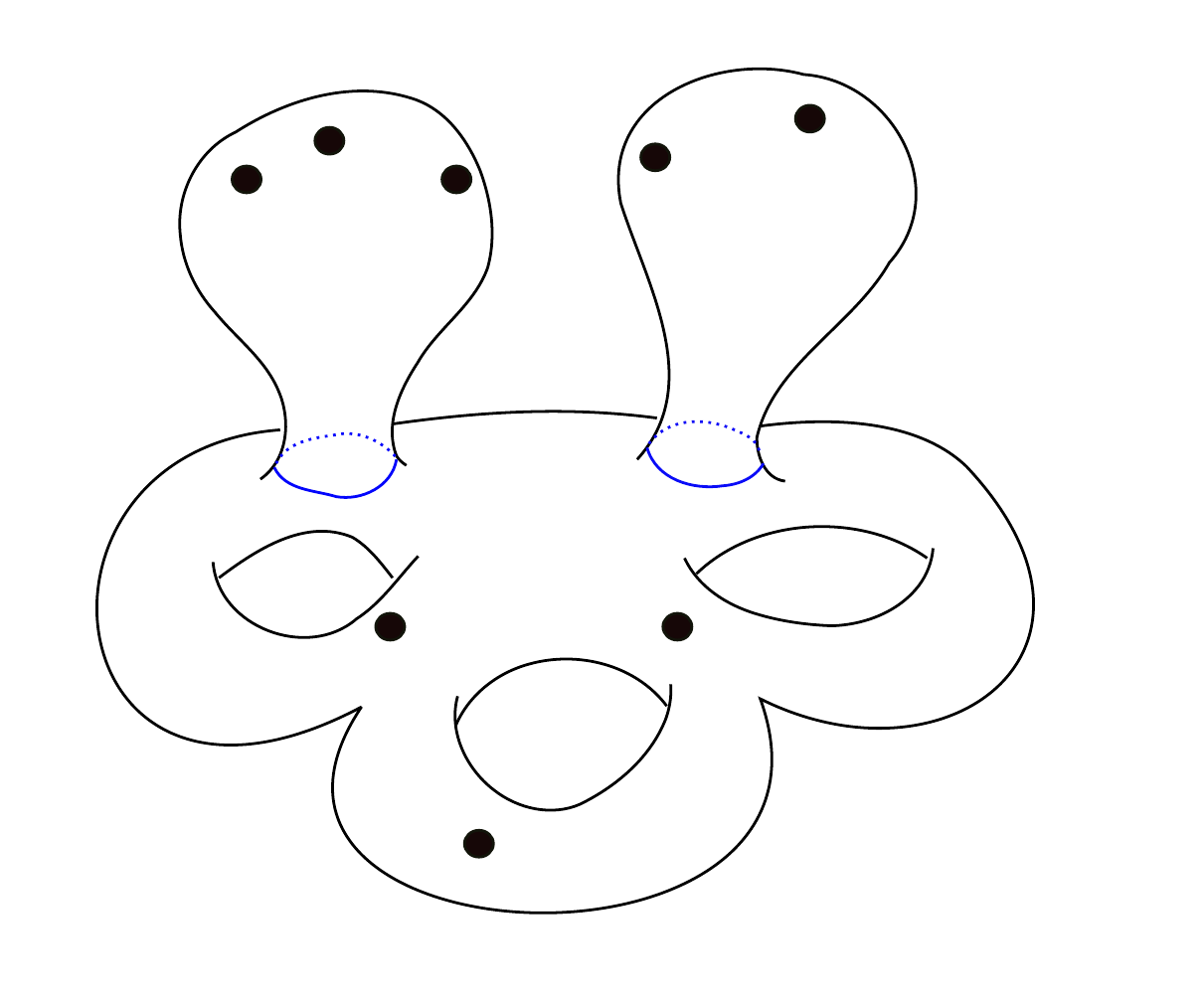}
\put (75, 57){$X_t$}
\put (10, 57){$Y_t$}
\put (10,16){$C_t$}
\put (19,64){$p_1$}
\put (26,67){$p_2$}
\put (36,64){$p_3$}

\put (57,40){$v_1$}
\put (27,40){$v_2$}

\put (54,66){$p_7$}
\put (66,69){$p_8$}

\put (35,30){$p_4$}
\put (34,13){$p_5$}
\put (51,30){$p_6$}

\put (-5,-0){\footnotesize{Figure 6: Canonical divisor of type $\kappa=(d_1,d_2,d_3,1^4,-1)$ on a smooth curve. }}
\end{overpic}
\end{center}
\end{figure}
Hence Figure $5$ gives the two possible level graphs to provide smoothable twisted canonical divisors of this type. Graph A gives the condition
\begin{equation*}
\res_x(\omega_X)+\res_y(\omega_Y)=0,
\end{equation*}
while Graph B gives the condition
\begin{equation*}
\res_y(\omega_Y)=0.
\end{equation*}
Hence all configurations of $y,p_1,p_2,p_3$ on the rational tail $Y$ are smoothable.

The necessity of the condition on the residue in this case can be again observed by considering topologically a family of twisted canonical divisors on smooth curves degenerating to a nodal marked curve. This situation is depicted in Figure 6. The different possible global residue conditions distinguish the relative speed at which $v_1$ and $v_2$ degenerate to nodes. Graph A in Figure 5 represents the situation where $v_1$ and $v_2$ degenerate to nodes at the same speed, while Graph B in Figure 5 represents the situation that $v_1$ degenerates faster than $v_2$, and hence due to the poles on $Y_t$, this results in no residue condition at $x$. 
\end{ex}

\subsection{Divisor theory on $\Mgnb$}
Let $\lambda$ denote the first Chern class of the Hodge bundle on $\Mgnb$ and $\psi_i$ denote the first Chern class of the cotangent bundle on $\Mgnb$ associated to the $i$th marked point where $1\leq i\leq n$. These classes are extensions of classes defined on $\Mgn$ that generate $\Pic(\Mgn)\otimes\QQ$. The boundary $\Delta=\Mgnb-\Mgn$ of $\Mgnb$ is codimension one. Define $\Delta_0$ as the locus of curves in $\Mgnb$ with a non-separating node. Define $\Delta_{i:S}$ for $0\leq i\leq g$, $S\subseteq \{1,...,n\}$ as the locus of curves with a separating node, separates the curve into a genus $i$ component containing the marked points from $S$ and a genus $g-i$ component containing the marked points from $S^c$, the complement of $S$. Hence we require $|S|\geq 2$ for $i=0$ and $|S|\leq n-2$ for $i=g$ and observe that $\Delta_{i:S}=\Delta_{g-i:S^c}$. These boundary divisors are irreducible and can intersect each other and self-intersect. The class of $\Delta_{i:S}$ in $\Pic(\Mgnb)\otimes\QQ$ is denoted by $\delta_{i:S}$. 

These divisor classes freely generate $\Pic(\Mgnb)\otimes \QQ$ for $g\geq 3$. For $g=2$ the classes $\lambda,\delta_0$ and $\delta_1$ generate $\Pic(\overline{\mathcal{M}}_2)\otimes \QQ$ with the relation
\begin{equation*}
\lambda=\frac{1}{10}\delta_0+\frac{1}{5}\delta_1.
\end{equation*}
Pulling back this relation under the map $\varphi:\overline{\mathcal{M}}_{2,n}\longrightarrow \overline{\mathcal{M}}_{2}$ that forgets the $n$ marked points we obtain
\begin{equation*}
\lambda=\frac{1}{10}\delta_0+\frac{1}{5}\sum_{1\in S}\delta_{1:S}.
\end{equation*}
$\Pic(\overline{\mathcal{M}}_{2,n})\otimes \QQ$ is freely generated by $\lambda,\psi_i$ and $\delta_{i:S}$ with this relation.~\cite{AC}\cite{HarrisMorrison} give an introduction to the divisor theory of $\Mgnb$.


\section{Divisors from the strata of abelian differentials}\label{divisorclass}
\begin{defn}
The divisor $D^n_\kappa$  in $\Mgnb$ for $\kappa=(k_1,...,k_m)$ with $\sum k_i=2g-2$ is defined as 
\begin{equation*}
D^n_\kappa=\overline{\{[C,p_1,...,p_n]\in\Mgn\hspace{0.15cm}|\hspace{0.15cm} [C,p_1,...,p_m]\in\mathcal{M}_{g,m} \text{ with  }  \sum k_ip_i\sim K_C   \}},
\end{equation*}
where $m=n+g-2$ or $n+g-1$ for holomorphic and meromorphic signature $\kappa $ respectively. When all $k_i$ are even this divisor has two irreducible components based on spin structure that we denote by the indices \emph{odd} and \emph{even}. Hence $D^n_\kappa$ is proportional to the pushdown of $\overline{\PPP}(\kappa)$ under the map forgetting the last $g-2$ or $g-1$ points in the holomorphic or meromorphic cases respectively. We use this notation to denote both the divisor and the class.
\end{defn}

For holomorphic signature $\kappa=(d_1,...,d_n,1^{g-2})$ with $d_i>0$, the divisors $D^n_\kappa$ correspond to generalised Brill-Noether divisors which \cite{Logan} used to investigate the Kodaira dimension of $\Mgnb$. The divisor $D^n_{\kappa}$ in $\overline{\mathcal{M}}_{g,g}$ for the signature $\kappa=(1^{2g-2})$ is extremal as it is contracted by the map to the universal Picard variety of degree $g$ line bundles~\cite{FarkasVerraU}. The closure of the anti-ramification locus or the divisor $D^n_{\kappa}$ in $\overline{\mathcal{M}}_{g,g-1}$ for the signature $\kappa=(1^{2g-4},2)$ was shown to be extremal by the construction of a covering curve with negative intersection~\cite{FarkasVerraTheta}.

The divisors of interest to us in this paper are those from meromorphic signature $\kappa$. In this section we summarise the known classes of these divisors. When all unmarked points are simple zeros the class was computed by~\cite{Muller}\cite{GruZak}\footnote{Note that in this formula $S\ne\{1,...,n\}$. In this case the coefficient is found by $S^c=\emptyset\subset S_+$. The condition on $S$ in the formula is separating the cases where all poles lie on the same component.}. For $\underline{d}=(d_1,...,d_n)$ with $\sum d_i=g-1$ and at least one $d_i<0$,
\begin{equation*}
D^n_{\underline{d},1^{g-1}}=-\lambda+\sum_{j=1}^n\begin{pmatrix}d_j+1\\2  \end{pmatrix}\psi_j
-0\cdot\delta_0-\sum_{\tiny{\begin{array}{cc}i,S\\S\subset S_+  \end{array}}}\begin{pmatrix}|d_S-i|+1\\2  \end{pmatrix}\delta_{i:S}
-\sum_{\tiny{\begin{array}{cc}i,S\\S\nsubset S_+  \end{array}}}\begin{pmatrix}d_S-i+1\\2  \end{pmatrix}\delta_{i:S}
\end{equation*}
in $\Pic(\Mgnb)\otimes \QQ$, where $S_+:=\{j\hspace{0.1cm}|\hspace{0.1cm}d_j>0\}$ and $d_S:=\sum_{j\in S}d_j$. 

In~\cite{Mullane2} the author generalises these to obtain for any fixed $g\geq 2$ and $n\geq 2$, a number of new infinite families of effective divisors from the strata of abelian differentials. The divisors relevant to our discussion are from meromorphic signatures with $n\geq3$. When the signature of unmarked points is $(2^{g-1})$ we obtain the \emph{coupled partition divisors}.
For $\underline{d}=(d_1,...,d_{n})$ such that $\sum d_i=0$ with  $\underline{d}^-\ne \{-2\}$, then
\begin{equation*}
D^{n}_{\underline{d},2^{g-1}}=2^{g-2}(2^{g+1}\lambda +2^{g-1}\sum_{j=1}^nd_j^2\psi_j-2^{g-2}\delta_0   
 -\sum_{\tiny{\begin{array}{cc}|d_S|=0\\|S|\ne n  \end{array}}}\sum_{i=0}^{g}2^{g-i+1}(2^i-1)\delta_{i:S}   
 -2^{g-1}\sum_{|d_S|\geq1}\sum_{i=0}^{g-1}d_S^2\delta_{i:S}        
 ).
\end{equation*}
If all $d_j$ are even then
\begin{equation*}
D^{n,\textit{odd}}_{\underline{d},2^{g-1}}=2^{g-2}((2^g-1)\lambda+\frac{2^g-1}{4}\sum_{j=1}^n d_j^2\psi_j-2^{g-3}\delta_0
-\sum_{\tiny{\begin{array}{cc}|d_S|=0\\|S|\ne n  \end{array}}}\sum_{i=0}^{g}(2^i-1)(2^{g-i}+1)\delta_{i:S}   -\frac{2^g-1}{4}\sum_{|d_S|\geq 2}\sum_{i=0}^{g-1}d_S^2\delta_{i:S}         )
\end{equation*}
and 
\begin{equation*}
D^{n,\textit{even}}_{\underline{d},2^{g-1}}=2^{g-2}((2^g+1)\lambda +\frac{2^g+1}{4}\sum_{j=1}^n d_j^2\psi_j -2^{g-3}\delta_0
-\sum_{|d_S|=0}\sum_{i=0}^{g}(2^i-1)(2^{g-i}-1)\delta_{i:S}     -\frac{2^g+1}{4}\sum_{|d_S|\geq 2}\sum_{i=0}^{g-1}d_S^2\delta_{i:S}         ).
\end{equation*}
For $\underline{d}=(-2,1,1)$, 
\begin{eqnarray*}
D^{3}_{\underline{d},2^{g-1}}&=&2^{g-3}(2^{g+1}\lambda +2^{g+2}\psi_1+2^{g-1}(\psi_2+\psi_3)-2^{g-2}\delta_0 -\sum_{i=0}^{g-1}2^{i+1}(2^{g-i}-1)\delta_{i\{1,2,3\}}\\
&&-\sum_{i=0}^{g-1}2^{i+1}(2^{g-i}+1)\delta_{i:\{2,3\}}
-\sum_{i=0}^{g-1}2^{g-1}(\delta_{i:\{1,2\}}+\delta_{i:\{1,3\}}).
\end{eqnarray*}

When the signature of the unmarked points is $(1^{g-2},2)$ we obtain the \emph{pinch partition divisors}. For $\underline{d}=(d_1,...,d_n)$ with $\sum d_i=g-2$, $d_j\leq -2$ and $d_i\geq0$ for $i\ne j$, then for $d_j\leq -3$,
\begin{eqnarray*}
D^{n}_{\underline{d},1^{g-2},2}&=&(26-4g)\lambda +\sum_{i=1}^n 2d_i((g-1)d_i+g-2)\psi_i-2\delta_0    +\sum_{i=0}^{g-1}c_{i:S}\delta_{i:S}
\end{eqnarray*}
where for $j\nin S$ and $d_S\leq i-1$,
\begin{equation*}
c_{i:S}=(2-2g) d_S^2+2 ( 2g i+g-4i+1) d_S-2 (g i^2+g i-3i^2 +i+1)
\end{equation*}
and for $d_S\geq i$,
\begin{equation*}
c_{i:S}=(2-2g) d_S^2  +2 (2g i-g-4i+2) d_S -2 (g i^2-3i^2-g i+4i).
\end{equation*}
For $d_j=-2$
\begin{eqnarray*}
D^{n}_{\underline{d},1^{g-2},2}&=&(27-4g)\lambda +4g\psi_j+\sum_{i\ne j} \frac{(4g(d_i+1)-5d_i-9)d_i}{2}\psi_i-2\delta_0    +\sum_{i=0}^{g-1}c_{i:S}\delta_{i:S}
\end{eqnarray*}
where for $j\nin S$ and $d_S\leq i-1$,
\begin{equation*}
c_{i:S}=\frac{1}{2}((5-4g)d_S^2+(8g i+4g-18i+3)d_S-4g i^2-4gi+13i^2-3i-4)
\end{equation*}
and for $d_S\geq i$,
\begin{equation*}
c_{i:S}=\frac{1}{2}((5-4g)d_S^2+(8g i-4g-18i+9)d_S-4g i^2+4g i+13i^2-17i).
\end{equation*}
In the case that $g=2$ the pinch partition and coupled partition divisors correspond. If further, all $d_i$ are even, the coupled partition divisor formula gives the classes of each of the irreducible components based on spin structure.

\section{Moving curves}\label{moving}
A moving curve $B$ of a projective variety $X$ is a curve with $B\cdot D\geq 0$ for all effective divisors $D$. Hence any effective divisor $D$ with zero intersection with a moving curve must lie on the boundary of the effective cone as $B\cdot(D-\epsilon A)<0$ for any ample divisor $A$ with $\epsilon>0$. One way to obtain moving curves is by fibrations. If the numerical equivalence classes of $B$ cover a Zariski dense subset of $X$ and the general curve is irreducible then $B$ is a moving curve.

Fix a meromorphic signature $\kappa$ of length $m\geq g+1$. In \S\ref{strata:abelian} we introduce the subvariety $\overline{\PPP}({\kappa})$ of $\overline{\mathcal{M}}_{g,m}$ of codimension $g$. If $\varphi:\overline{\PPP}({\kappa})\longrightarrow \overline{\mathcal{M}}_{g,m-g-1}$ forgets the first $g+1$ points, we obtain a fibration of $\overline{\PPP}({\kappa})$ with one dimensional fibres. Consider the curve $B^n_{\kappa}$ defined as
\begin{equation*}
B^n_{\kappa}:=\overline{\bigl\{ [C,p_1,...,p_n]\in\mathcal{M}_{g,n} \hspace{0.15cm}\big| \hspace{0.15cm} \text{fixed general $[C,p_{g+2},...,p_m]\in\mathcal{M}_{g,m-g-1}$ and $[C,p_{1},...,p_m]\in\PPP(\kappa)$}    \bigr\}   }.
\end{equation*}
If $\pi:\overline{\PPP}({\kappa})\longrightarrow \overline{\mathcal{M}}_{g,n}$ forgets all but the first $n$ points, then for $n\geq g+1$ we obtain $\pi_*\varphi^*[C,p_{g+2},...,p_m]=B^n_{\kappa}$. When $n\leq g$ we have $\pi_*\varphi^*[C,p_{g+2},...,p_m]=dB^n_{\kappa}$ for some positive integer $d$.

\begin{prop}
The class of $B^n_{\kappa}$ is a moving curve in $\Mgnb$ when $|\kappa|\geq n+g$ and $\PPP(\kappa)$ is irreducible.
\end{prop}

\begin{proof}
For any general $[C,p_1,...,p_n]\in\mathcal{M}_{g,n}$ to be in a numerical equivalence class of a curve $B^n_\kappa$ we require some $p_{n+1},...,p_m$ such that 
\begin{equation*}
\sum_{i=1}^mk_ip_i\sim K_C.
\end{equation*}
Let $d=\sum_{i=1}^nd_i$ and consider the map
\begin{eqnarray*}
\begin{array}{cccccc}
f:C^{m-n}&\longrightarrow &\Pic^{d}(C)\\
(p_{n+1},...,p_m)&\longmapsto&K_C(-\sum_{i=n+1}^md_ip_i).
\end{array}
\end{eqnarray*}
The domain has dimension $m-n$, while the target has dimension $g$. Clearly $f^{-1}(\sum_{i=1}^nd_ip_i)$ will be non-empty for general $[C,p_1,...,p_n]$ when $m\geq n+g$. Hence in these cases curves with numerical equivalence class equal to $B^n_\kappa$ cover a Zariski dense subset of $\Mgnb$.

If a general curve with numerical class $B^n_{\kappa}$ is irreducible then $B^n_{\kappa}$ is a moving curve. However, if the general $B^n_{\kappa}$ is reducible, it can also be the class of a moving curve if all components of $B^n_{\kappa}$ have the same class and are hence proportional to $B^n_{\kappa}$. 

In our case if a general curve with numerical class $B^n_{\kappa}$ is reducible with components with different class, taking the closure of these distinguishable components over all $[C,p_{n+1},...,p_m]\in\mathcal{M}_{g,m-n}$ would contradict the irreducibility of  $\PPP(\kappa)$.
\end{proof}

\boundary*

\begin{proof}
Observe $|\kappa|=n+g-1$. Fix a general curve $C$ and general points $p_{g+2},...,p_{n+g+1}$. If any point in the curve $B^n_{\kappa,1,-1}$ is also in the divisor $D^n_{\kappa}$ we require some $q_1,...,q_{g-1}\in C$ such that
\begin{equation*}
\sum_{i=1}^{n+g-1}k_ip_i   +p_{g+n}-p_{g+n+1}\sim \sum_{i=1}^{n}k_ip_i+\sum_{i=1}^{g-1}k_{i+n}q_i\sim K_C.
\end{equation*}
Hence we require
\begin{equation*}
\sum_{i=n+1}^{n+g-1}k_ip_i+p_{g+n}-p_{g+n+1}-\sum_{i=1}^{g-1}k_{i+n}q_i\sim \OO_C.
\end{equation*}
But this implies that $C$ has a cover of $\PP^1$ of degree $d=1+\sum|k_i| $ with ramification profile above $0$ and $\infty$ given by the points with positive and negative coefficients respectively in the equation above. Riemann-Hurwitz implies that in a general such cover there will be $4g-2$ other simple ramification points. Hence the dimension of the space of such covers is $4g-3$, the dimension of the moduli of $4g$ points in $\PP^1$. But as $g+1$ points are fixed general and $\dim(\mathcal{M}_{g,g+1})=4g-2$ there does not exist such $q_i$ for a general choice of $p_i$ for $i=n+1,...,g+n+1$.
\end{proof}

\section{Covering curves and the effective cone of $\Mgnb$}\label{covering}
An effective divisor $D$ in a projective variety $X$ is \emph{extremal} or spans an extremal ray in the pseudo-effective cone if the divisor class $D$ cannot be written as a sum $m_1D_1+m_2D_2$ of pseudo-effective $D_i$ with $m_1,m_2>0$ unless $D, D_1$ and $D_2$ are all proportional classes. An effective divisor $D$ is \emph{rigid} if $h^0(mD)=\dim H^0(mD)=1$ for every positive integer $m$. 

A curve $B$ contained in an effective divisor $D$ is known as a \emph{covering curve} for $D$ if irreducible curves with numerical class equal to $B$ cover a Zariski  dense subset of $D$. Negative intersection by a covering curve is a well-known criterion for an irreducible effective divisor to be extremal and rigid. 

\begin{lem}\label{lemma:covering}
If $B$ is a covering curve for an irreducible effective divisor $D$ with $B\cdot D<0$ then $D$ is extremal and rigid.
\end{lem}
\begin{proof}
~\cite[Lemma 4.1]{ChenCoskun}.  
\end{proof}

\begin{rem}
Negative intersection of a covering curve of an irreducible divisor $D$ provides further information about the structure of the pseudo-effective cone near $D$. It implies every pseudo-effective divisor in some neighbourhood of the extremal ray $D$ can be written as $\varepsilon D+ \varepsilon'D'$ for $\varepsilon>0,\varepsilon'\geq0$ and $D'$ pseudo-effective. Hence further to being extremal and rigid, the boundary of the pseudo-effective cone is locally polygonal at such a divisor $D$ and not rounded. ~\cite[\S6]{Opie} provides a proof and discussion of this fact.
\end{rem}

\begin{prop}\label{prop:covering}
If $D^n_\kappa$ is an irreducible divisor then $B^n_\kappa$ is the class of a covering curve. If $D^n_\kappa$ has two components based on spin structure then $B^{n,\text{odd}}_\kappa$ and $B^{n,\text{even}}_\kappa$ are the class of a covering curve for the odd and even components respectively.
\end{prop}

\begin{proof}
Clearly the numerical classes of $B^n_\kappa$ cover a Zariski dense subset of $D^n_\kappa$. If the general curve $B^n_\kappa$ is reducible and components have different classes, then taking the closure of one of these components over all numerical classes of $B^n_\kappa$ would contradict the irreducibility of $D^n_\kappa$. Hence irreducible curves with class proportional to $B^n_\kappa$ cover a Zariski dense subset of $D^n_\kappa$.

This same proof holds in the case of spin structure components.
\end{proof}

By observing that for some signatures $\kappa$ the curve $B^n_\kappa$  forms a component of a specialisation of $B^n_{\kappa,-1,1}$ we obtain the negative intersection of the covering curve for the corresponding divisors $D^n_\kappa$.

\begin{prop}\label{prop:intersection}
$B^{g+1}_{\underline{d},1^{g-1}}\cdot D^{g+1}_{\underline{d},1^{g-1}}<0$ 
for $g\geq 2$ and $\underline{d}=(d_1,d_2,d_3,1^{g-2})$ with $d_1+d_2+d_3=1$ and $\sum_{d_i<0}d_i\leq-2$.
\end{prop}

\begin{proof}
Theorem~\ref{thm:boundary} implies
\begin{equation*}
B^{g+1}_{\underline{d},1^{g-1},1,-1}\cdot D^{g+1}_{\underline{d},1^{g-1}}=0
\end{equation*}
As presented in \S\ref{moving}, curves linearly equivalent to $B^{g+1}_{\underline{d},1^{g-1},1,-1}$ are constructed by a fibration of $\overline{\PPP}(\underline{d},1^{g-1},1,-1)$ over $\overline{\mathcal{M}}_{g,g+1}$. Consider the special fibre obtained for $[C,p_{g+2},...,p_{2g+2}]\in\overline{\mathcal{M}}_{g,g+1}$ with $p_{g=2},...,p_{2g}$ general points in $C$ while $p_{2g+1}$ and $p_{2g+2}$ sit on a $\PP^1$ tail we denote $X$ above a general point $x$ on the curve $C$. To find the components of the curve we describe the fibre by finding all possible limit canonical divisors of this type.

\begin{figure}[htbp]
\begin{center}
\begin{overpic}[width=1\textwidth]{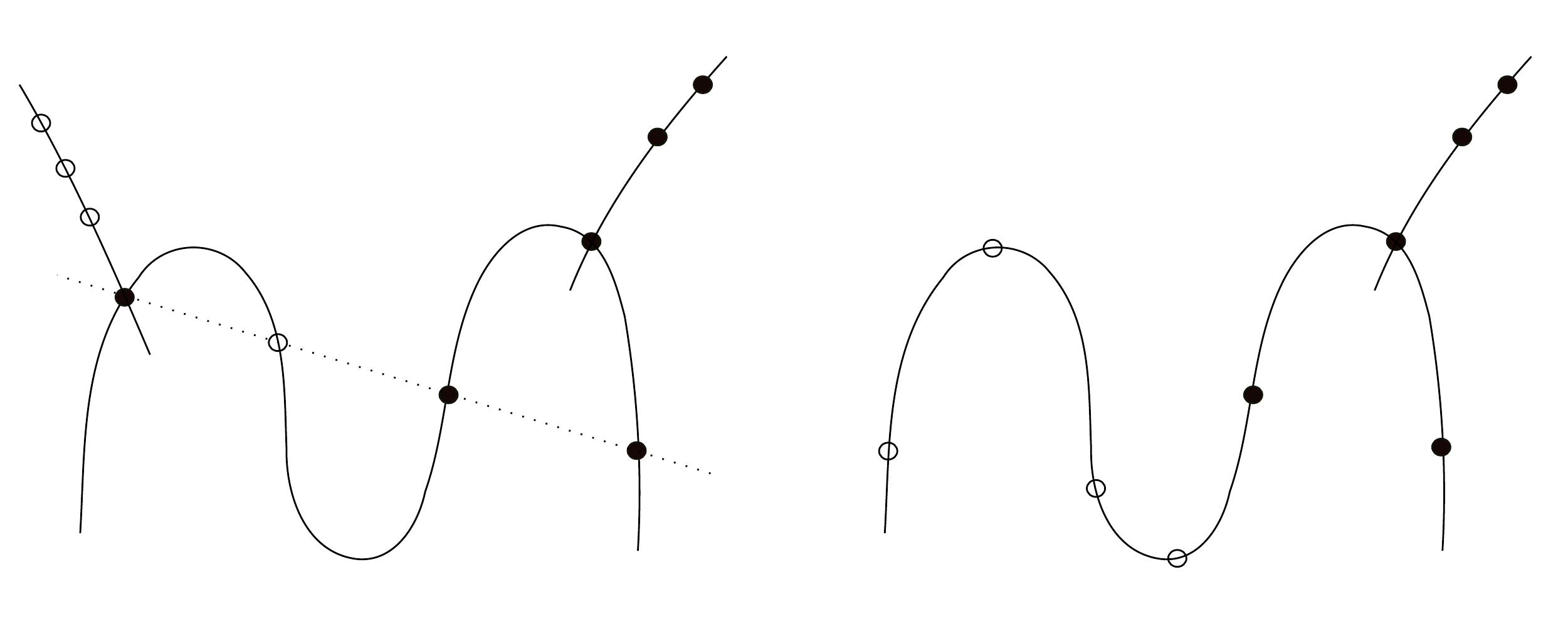}
\put (2,8){$C$}
\put (10,30){$Y\cong \PP^1$}
\put (4,33){$p_1$}
\put (5.5,30){$p_2$}
\put (7,26.5){$p_3$}
\put (5,20.5){$y$}
\put (15,17){$p_4$}
\put (25.5,14){$p_5$}
\put (38,10){$p_6$}
\put (41,27){$X\cong \PP^1$}
\put (39,32){$p_7$}
\put (41.5,35){$p_8$}
\put (37,27){$x$}
\put (13,0){$y+p_4+p_5+p_6\sim K_C$}

\put (53,8){$C$}
\put (58,11.5){$p_1$}
\put (63,27){$p_2$}
\put (66.5,9){$p_3$}
\put (74.5,7){$p_4$}
\put (76.5,14){$p_5$}
\put (89,10){$p_6$}
\put (92,27){$X\cong \PP^1$}
\put (90,32){$p_7$}
\put (92.5,35){$p_8$}
\put (88,27){$x$}
\put (55,0){$d_1p_1+d_2p_2+d_3p_3+p_4+p_5+p_6\sim K_C$}

\put (13,-7){\footnotesize{Figure 7: Two components of $B^4_{d_1,d_2,d_3,1^4,-1}$ in $\overline{\mathcal{M}}_{3,4}$ when $p_7$ and $p_8$ sit on a $\PP^1$ tail.}}
\end{overpic}
\vspace{0.6cm}
\end{center}
\end{figure}

The limit canonical divisors of interest are of two different types. We denote by $\delta(B^{g+1}_{\underline{d},1^{g-1},1,-1})$ the curve contained in the boundary of $\overline{\mathcal{M}}_{g,g+1}$. This is the curve created by three $p_i$ moving freely on a $\PP^1$-tail. Consider $d_i\ne 1$ for $i=1,2,3$. In this case $\delta(B^{g+1}_{\underline{d},1^{g-1},1,-1})$  is the curve created by the points $p_1,p_2,p_3$ moving freely on a $\PP^1$-tail we denote $Y$ attached to the curve $C$ at a point $y$. The resulting twisted canonical divisor is of the form
\begin{eqnarray*}
y+\sum_{i=4}^{2g}p_i &\sim& K_C\\
p_{2g+1}-p_{2g+2}-2x&\sim& K_X\\
d_1p_1+d_2p_2+d_3p_3-3y&\sim& K_Y.
\end{eqnarray*}
The pointed nodal curve in the left of Figure 7 shows this situation for $g=3$ which is presented in detail in Example~\ref{Ex2}. In the general genus case on a general curve $C$ we have $h^0(K_C-p_{g+2}-...-p_{2g})=1$ providing $(g-1)!$ solutions for $y,p_4,...,p_{g+1}$. We are left to check the global residue condition to show that such twisted canonical divisors are smoothable. Observe that due to the simple pole at $p_{2g+2}$, for any differential $\omega_X$ on $X$ with $(\omega_X)\sim p_{2g+1}-p_{2g+2}-2x$ we have $\res_x(\omega_X)=-\res_{p_{2g+2}}(\omega_X)\ne 0$. Hence Figure 8 gives the two possible level graphs to provide smoothable twisted canonical divisors of this type. Graph A gives the condition
\begin{equation*}
\res_x(\omega_X)+\res_y(\omega_Y)=0,
\end{equation*}
while Graph B gives the condition
\begin{equation*}
\res_y(\omega_Y)=0.
\end{equation*}
Hence all configurations of $y,p_1,p_2,p_3$ on the rational tail $Y$ are included. In the case that some $d_i=1$, say $d_3=1$, we have  $\delta(B^{g+1}_{\underline{d},1^{g-1},1,-1})$ will include more components where $p_1,p_2$ and $p_i$ move freely on a $\PP^1$-tail for $i=3,...,g+1$.
\begin{figure}[htbp]
\begin{center}
\vspace{1cm}
\begin{overpic}[width=0.7\textwidth]{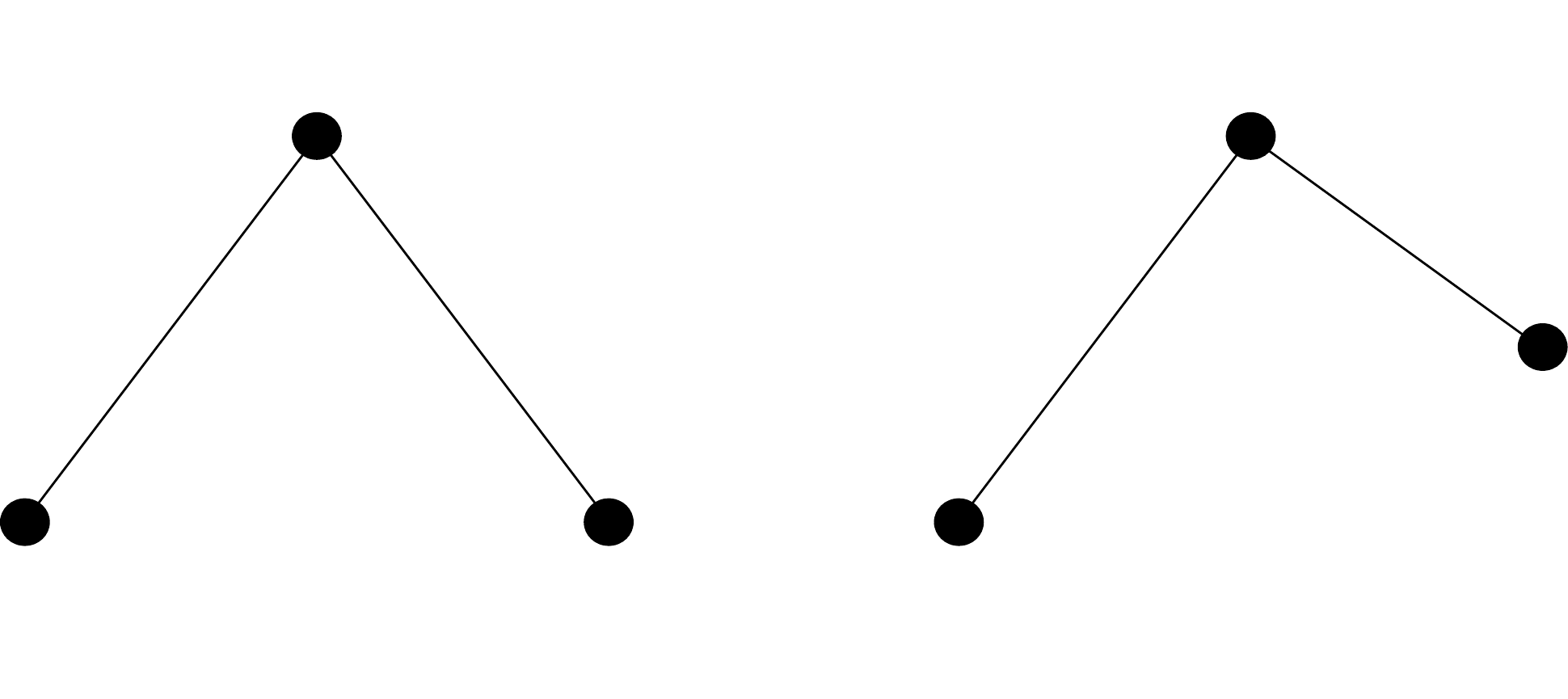}

\put (19,40){$C$}
\put (0,6){$X$}
\put (38,6){$Y$}
\put (13,47){Graph A}

\put (78.5,40){$C$}
\put (59,6){$X$}
\put (97,17){$Y$}
\put (72,47){Graph B}

\put (11,-0){\footnotesize{Figure 8: Level graphs giving the global residue condition.}}
\end{overpic}
\end{center}
\end{figure}

The other candidate for limit twisted canonical divisor in this special fibre are twisted canonical divisors of the form
\begin{eqnarray*}
d_1p_1+d_2p_2+d_3p_3+\sum_{i=4}^{2g}p_i &\sim& K_C\\
p_{2g+1}-p_{2g+2}-2x&\sim& K_X,
\end{eqnarray*}
where the points $p_1,...,p_{g+1}$ vary in a one dimensional family in the curve $C$. The pointed curve in the right of Figure 7 shows this situation for $g=3$. As there are poles on both components the global residue condition is empty and all solutions of this type are smoothable. Hence this component is simply the curve $B^{g+1}_{{\underline{d},1^{g-1}}}$.

We obtain
\begin{equation*}
B^{g+1}_{\underline{d},1^{g-1},1,-1}\sim \delta(B^{g+1}_{\underline{d},1^{g-1},1,-1})+B^{g+1}_{\underline{d},1^{g-1}}.
\end{equation*}
Observe $\delta(B^{g+1}_{\underline{d},1^{g-1},1,-1})$ intersects $D^{g+1}_{\underline{d},1^{g-1}}$ at the finite set of points where the residue at $y$ is zero as the level graph for these solutions contains only two components with no poles on $C$. In Example~\ref{Ex1} this situation is explained in detail. This implies
\begin{equation*}
 \delta(B^{g+1}_{\underline{d},1^{g-1},1,-1})\cdot D^{g+1}_{\underline{d},1^{g-1}}>0
\end{equation*} 
and hence
\begin{equation*}
 B^{g+1}_{\underline{d},1^{g-1}}\cdot D^{g+1}_{\underline{d},1^{g-1}}<0.
\end{equation*} 
\end{proof}

\extremal*

\begin{proof}
$\PPP(d_1,d_2,d_3,1^{2g-3})$ is irreducible, hence the divisor $D^{g+1}_{\underline{d},1^{g-1}}$ is irreducible and Proposition~\ref{prop:covering} and Proposition~\ref{prop:intersection} give a covering curve for the divisors with negative intersection. Hence by Lemma~\ref{lemma:covering} the divisors are extremal and rigid.

For $n>g+1$ where $\varphi:\Mgnb\longrightarrow \overline{\mathcal{M}}_{g,g+1}$ forgets all but the first $g+1$ points we have $\varphi^*D^{g+1}_{\underline{d},1^{g-1}}$ is irreducible as the pullback of an irreducible divisor. Further $B^{n}_{\underline{d},0^{n-g-1},1^{g-1}}$ provides a covering curve for $\varphi^*D^{g+1}_{\underline{d},1^{g-1}}$ with 
\begin{equation*}
\varphi_*B^{n}_{\underline{d},0^{n-g-1},1^{g-1}}=B^{g+1}_{\underline{d},1^{g-1}}
\end{equation*}
and by the projection formula
\begin{equation*}
B^{n}_{\underline{d},0^{n-g-1},1^{g-1}}\cdot \varphi^*D^{g+1}_{\underline{d},1^{g-1}}=\varphi_*B^{n}_{\underline{d},0^{n-g-1},1^{g-1}}\cdot D^{g+1}_{\underline{d},1^{g-1}}=B^{g+1}_{\underline{d},1^{g-1}}\cdot D^{g+1}_{\underline{d},1^{g-1}}<0.
\end{equation*}
Hence $\varphi^*D^{g+1}_{\underline{d},1^{g-1}}$ are extremal and rigid.
\end{proof}

This strategy is a generalisation of that used by~\cite{ChenCoskun} in genus $g=1$. As a result of the translation automorphism on any elliptic curve, for $g=1$ this strategy provides a covering curve with negative intersection for the irreducible divisors $D^3_{d_1,d_2,-d_1-d_2}$ in $\overline{\mathcal{M}}_{1,3}$ for $\gcd(d_1,d_2)=1$.


A natural question to ask is whether this behaviour of multiple components in the special fibre of a fibration can be used to show the two spin structure components of a reducible divisor $D^n_\kappa$ to be extremal and rigid. Consider $D^{g+1}_\kappa=D^{g+1,\text{odd}}_\kappa+D^{g+1,\text{even}}_\kappa$ for $\kappa=(d_1,...,d_{g+1},2^{g-1})$ for $d_i$ even and nonzero. Then
\begin{equation*}
B^{g+1}_{\kappa,1,-1}\cdot D^{g+1}_\kappa=0,
\end{equation*}
and further by considering the same specialisation of this curve we obtain
\begin{equation*}
B^{g+1}_{\kappa,1,-1}\sim B^{g+1,\text{odd}}_{\kappa}+B^{g+1,\text{even}}_{\kappa}.
\end{equation*}
Although these curves provide covering curves for the components of the divisor we unfortunately have
\begin{equation*}
B^{g+1,\text{even}}_\kappa\cdot D^{g+1,\text{odd}}_\kappa=B^{g+1,\text{odd}}_\kappa\cdot D^{g+1,\text{even}}_\kappa=0,
\end{equation*}
which yields 
\begin{equation*}
B^{g+1,\text{odd}}_\kappa\cdot D^{g+1,\text{odd}}_\kappa=B^{g+1,\text{even}}_\kappa\cdot D^{g+1,\text{even}}_\kappa=0.
\end{equation*}

A projective variety $X$ is a Mori dream space~\cite{HuKeel} if the Cox ring of $X$ is finitely generated and the Neron-Severi group of $X$ is $\Pic(X)\otimes\QQ$. In this case the nef and semi-ample cones of divisors on $X$ correspond and $\Eff(X)$ is rational polyhedral and has a specific chamber decomposition. Restricting to the open cases in $g=2$, Theorem~\ref{thm:extremal} implies the following.

\moridreamspace*


\bibliographystyle{plain}
\bibliography{base}
\end{document}